\documentclass{article}
\usepackage[utf8]{inputenc}
\usepackage{authblk}
\usepackage{setspace}
\usepackage[margin=1.25in]{geometry}
\usepackage{graphicx}
\graphicspath{ {./figures/} }
\usepackage{subcaption}
\usepackage{amsmath}
\usepackage{amssymb}
\usepackage{lineno}
\usepackage{comment}
\usepackage{amsthm}
\usepackage{physics}
\newtheorem{theorem}{Theorem}
\usepackage{amsfonts} 
\usepackage[font={small,it}]{caption}
\usepackage{url}
\usepackage{hyperref}
\usepackage{xcolor}
\hypersetup{
    colorlinks,
    linkcolor={red!50!black},
    citecolor={blue!50!black},
    urlcolor={blue!80!black}
}
\usepackage[natbib=true,backend=biber,sorting=none,style=ieee, citestyle=numeric-comp]{biblatex}

\usepackage{pdfpages}
\title{A Study of Qualitative Correlations Between Crucial Bio-markers and the Optimal Drug Regimen of Type-I Lepra Reaction: A Deterministic Approach }

\author[1]{Dinesh Nayak} 
\author[2]{A.V. Sangeetha}
\author[2]{D. K. K. Vamsi}

\affil[1, 3]{ \ Department of Mathematics and Computer Science, Sri Sathya Sai Institute of Higher Learning, Andhra Paresh, India - 515134}
\affil[2]{ \ Central Leprosy Teaching and Research Institute,Chengalpattu, Tamil Nadu, India - 603003 }
\affil[1]{First Author. Email: dineshnayak@sssihl.edu.in}
\affil[3]{Corresponding author. Email: dkkvamsi@sssihl.edu.in}
\begin{document}
\maketitle
\begin{abstract}
\textit{Mycobacterium leprae} is a bacteria that causes the disease \textit{Leprosy} (Hansen's disease), which is a neglected tropical disease. More than 200000 cases are  being reported per year world wide. This disease leads to a chronic stage  known as \textit{Lepra reaction} that majorly causes nerve damage of peripheral  nervous system leading to loss of organs.  The early detection of this \textit{Lepra reaction} through the level of bio-markers can prevent this reaction occurring and  the further disabilities. Motivated by this, we frame a mathematical model considering the pathogenesis of leprosy and the chemical pathways involved in  \textit{Lepra reactions}. The model incorporates the dynamics of the susceptible schwann cells, infected schwann cells and the bacterial load and the concentration levels of the bio markers $IFN-\gamma$, $TNF-\alpha$, $IL-10$, $IL-12$, $IL-15$ and $IL-17$. We consider a nine compartment optimal control problem considering the drugs used in Multi Drug Therapy (MDT) as controls. We validate the model using 2D - heat plots.  We study the correlation between the bio-markers levels and  drugs in MDT and propose a optimal drug regimen through these optimal control studies. We use  the \textit{Newton's Gradient 
Method } for the optimal control studies.  
\end{abstract}

\section{Introduction}
\hspace{0.2 in} 
\textit{Leprosy} the oldest disease known to human civilization, is one of the highly neglected tropical disease caused by  a slow growing bacteria called \textit{ Mycobacterium leprae} (M. leprae). Mainly this bacteria causes damage to the \textit{schwann cells} and hence the skin thereby the peripheral nervous system of host body gets impacted. It also has bad impact on eyes and mucosa of upper respiratory tract. The report of \textit{World Health Organisation} 
 (WHO) states that about 120 countries are still reporting new cases of Leprosy which  accounts to more than 200000 per year \cite{web2}.  In the year 2021 India alone spotted about 75, 395 new cases \cite{web3}. Leprosy disease can  lead to a chronic phase of Lepra reaction that causes permanent disabilities and loss of organs.  Early detection of the disease by observing the key changes in the bio-markers level will play a vital role to prevent the  losses. \\

\hspace{0.2 in} The chemical and metabolic properties of  the cytosol environment of   host cell that gets altered in the presence of the M.Leprae was first explained by \textit{Rudolf Virchow } (1821–1902) in the late nineteenth century \cite{virchow1865krankhaften}. Further different clinical studies explained about the path way of cytokine responses based upon which there are mainly two types of \textit{Lepra reactions}. The Type-1 \textit{Lepra reactions} are associated with cellular immune response where as Type 2 reactions are associated with humoral immune response  \cite{luo2021host, bilik2019leprosy}. Both these path ways involve the crucial  bio-markers/cytokines  such as $IFN-\gamma$, $TNF-\alpha$, $IL-10$, $IL-12$, $IL-15$ and $IL-17$ \cite{oliveira2005cytokines}.\\

\hspace{0.2 in} There are quite a number of biochemical studies that deal with the pathogenesis of \textit{Lepra reaction} \cite{ojo2022mycobacterium} and some on growth of the bacteria
and chemical consequences \cite{oliveira2005cytokines}. But very limited mathematical modeling  research is done  till date for this particular disease. Some studies dealing with the population level dynamics of the disease include \cite{blok2015mathematical, giraldo2018multibacillary}. The paper \cite{ghosh2021mathematical}  explores the cellular dynamics within the host. To our knowledge there is no work done yet dealing with the dynamics of the bio-markers involved in \textit{Lepra reactions.} Also there seems to be neither any clinical work  that clearly deals with the dynamics of  of bio-markers  during \textit{Lepra reactions.} Hence it is extremely important to study the dynamics at the bio-markers levels and their correlation  with the MDT drugs  that can help the clinicians for control of occurence of \textit{Lepra reactions.} \\

\hspace{0.2 in}Motivated by these observations in this work we propose to study the dynamics of the bio-markers  through chemical reactions.  In the next section we discuss and detail the proposed mathematical model.  We use the drugs in MDT as control variables. We  frame an optimal control problem along with a cost function  $\mathcal{J}_{min}$. In section $3$ we validate this model using the 2D -  heat plots. Further in the section $4$ we establish the existence of an optimal solution for the proposed optimal control problem. Next in the section $5$ we do the numerical studies. Initially we discuss the numerical scheme used namely the \textit{Newton's Gradient 
Method } for the optimal control studies.  Later we discuss  the inferences from the numerical simulations.  Finally in the last section  we do the  discussions and conclusions for this work.

\section{Mathematical model formulation}

 \hspace{0.2 in}Based on the clinical literature we conisder a model that consist of Susceptible schwann cells $S(t)$, Infected schwann cells $I(t), $  Bacterial Load $B(t)$ along with  five  cytokines  that play crucial role in \textit{Type-I Lepra reaction}. We have considered the cytokines $IFN-\gamma$, $TNF-\alpha$, $IL-10$, $IL-12$, $IL-15$ and $IL-17$ by capturing their concentration dynamics in \textit{Type-I Lepra reaction}. Below we discuss briefly each of the compartments  in the  model.\\

 \textbf{S(t) compartment:} The first term of the equation (\ref{sec11equ1})  deals with the natural birth rate of the susceptible schwann cell. The second term describes the decrease in number of susceptible cells  $S(t)$ at a rate $\beta$  due to infection by the bacteria (followed by the law of mass action). $\gamma$ represents death  of $S$ due to the cytokines response and $\mu_1$ represents the natural death rate of $S$. The rest of the terms account for the controls owing to the  due to MDT interventions.   \\

\textbf{I(t) compartment:} The increase of the  infected cells is accounted by the  term $\beta BS$ in equation (\ref{sec11equ2}). Decrease of these cells due to the cytokines response are at a rate $\delta$ and  the natural death rate is $\mu_1.$  The rest of the terms are associated with controls  owing to  MDT interventions.\\

\textbf{B(t) compartment:} The bacterial load increases  indirectly due to an  increase in $I(t)$ as the burst of more cells with bacteria increases their replication. This rate $\alpha$ is accounted in the first term of  (\ref{sec11equ3}).   $y$  denotes the rate of cleaning of $B(t)$ due to cytokines reposes and $\mu_2$ is the natural death rate of bacteria. The rest of the terms are associated with controls  owing to  MDT interventions.\\

\textbf{$I_{\gamma}$(t) compartment:} This compartment deals with the concentration level of Interferon-gamma ($IFN-\gamma$) through equation (\ref{sec11equ4}). The  first term represents the production of $IFN-\gamma$ in the presence of the infection. The second term deals with  the inhibition of the concentration level due other cytokines \cite{luo2021host, parida1993role} and the last term accounts for the  natural decay of the concentration.\\

\textbf{$T_{\alpha}$(t) compartment:} The  concentration of Tumour Necrosis Factor $\alpha$ ($TNF-\alpha$) is  increased by the inter action between $IFN-\gamma$ and the infected cells \cite{luo2021host, parida1993role}. This is being captured in the first term of the equation (\ref{sec11equ5}). The second term represents the natural decay. \\

Similar formulations are used for studying the concentration levels of  \textit{Interleukin-$12$}\textbf{($IL-12$)}, \textit{Interleukin-$15$}\textbf{($IL-15$)} and \textit{Interleukin-$17$}\textbf{($IL-17$)} compartments by equations (\ref{sec11equ7}),(\ref{sec11equ8}) and (\ref{sec11equ9}) respectively. \\

\textbf{$I_{10}$(t) compartment:} Equation (\ref{sec11equ6}) deals with  the  concentration levels of  \textit{Interleukin-$10$} \textbf{($IL-10$)}. The first term accounts for the production and the last term  for  the decay. The  middle term considers  the inhibition of $IL-10$ due to $IFN-\gamma$ \cite{luo2021host}. \\

{\bf{\large{Description of the Control Variables dealing with the MDT Drugs}}} \\

\hspace{0.2 in} WHO guidelines 2018 recommends a MDT for Leprosy,  that consist of three drugs  \textit{Rifampin, Dapsone} and \textit{Clofazimine} \cite{maymone2020leprosy, tripathi2013essentials}. The impact of each of these drugs and their mathematical articulation as control variables are as following.\\

\textbf{Rifampin ($D_{11},D_{12},D_{13}$):} \textit{Rifampin} is known for  rapid bacillary killing. Due to this there is an indirect decrease in the  amount of cells getting infected \cite{bullock1983rifampin}. Hence we incorporate  this with the control variable $D_{12}(t)$ in the compartment of infected cells of (\ref{sec11equ1}) -  (\ref{sec11equ9}) with a negative sign and in the bacterial load compartment this control is introduced as   $D_{13}^{2}(t)$.  Here the square on $D_{13}(t)$ is used to capture the extent of intense action on bacterial load.  This drug also reduces the susceptible cells, hence the control $D_{11}(t)$ is introduced with a negative sign. \\

\textbf{Dapsone ($D_{21},D_{22},D_{23}$):}  The drug {dapsone} is bactericidal and bacteriostatic against \textit{M. leprae}  \cite{paniker2001dapsone}. In a similar way to capture the drug action, we incorporate $D_{21}(t)$ and $D_{22}(t)$ in compartment $S$(equation (\ref{sec11equ1})) and compartment $I$(equation (\ref{sec11equ2})) respectively and $D_{23}^2(t)$ in the $B$(equation (\ref{sec11equ3})) compartment. \\

\textbf{Clofazimine ($D_{31},D_{33}$):} Clofazimine, the third drug in MDT for Leprosy acts as an immuno-suppressive and also causes  the static level of bacteria (bacteriostatic) against \textit{M. leprae } by binding with DNA of the bacteria and hence causing the inhibition of template function of DNA \cite{garrelts1991clofazimine}. Therefore to incorporate  this action, we include the control variable $D_{31}(t)$ in the $S(t)$ compartment in equation (\ref{sec11equ1}) resulting increase of these cells.  $D_{33}(t)$ is negatively incorporated  to account for the inhibition of bacterial replication.\\

\hspace{0.2 in} We next propose the   mathematical model dealing with the mechanism of drug action on each compartment of susceptible cells, infected cells and bacterial load along with concentration level of the cytokines.  \\

 \begin{eqnarray}
   	\frac{dS}{dt} &=&  \omega \ - \beta SB  - \gamma S - \mu_{1} S -D_{11}(t)S - D_{21}(t)S+D_{31}(t)S\label{sec11equ1} \\
   	\frac{dI}{dt} &=& \beta SB \ -\delta I - \mu_{1} I-D_{12}(t)I - D_{22}(t)I  \label{sec11equ2}\\ 
   	\frac{dB}{dt} &=&  \big(\alpha-D_{23}^2(t)-D_{33}(t)\big) I  \ - y B    -  \mu_{2} B-D_{13}^2(t)B       \label{sec11equ3}\\
        \frac{d I_{\gamma}}{dt} &=& \alpha_{I_\gamma} I -\big[ \delta^{I_{\gamma}}_ {T_\alpha}T_{\alpha}+\delta^{I_{\gamma}}_ {I_{12}}I_{12}+\delta^{I_{\gamma}}_ {I_{15}}I_{15}+\delta^{I_{\gamma}}_ {I_{17}}I_{17} \big]I-\mu_{I_\gamma} (I_{\gamma}- Q_{I_\gamma}) \label{sec11equ4}\\
        \frac{d T_{\alpha}}{dt} &=& \beta_{T_{\alpha}} I_{\gamma}I- \mu_{T_{\alpha}}(T_{\alpha}- Q_{T_\alpha}) \label{sec11equ5}\\
        \frac{d I_{10}}{dt} &=& \alpha_{I_{10}}I-\delta^{I_{10}}_{I_{\gamma}}I_{\gamma}-\mu_{I_{10}}(I_{10}- Q_ {I_{10}}) \label{sec11equ6}\\
        \frac{d I_{12}}{dt} &=& \beta_ {I_{12}} I_{\gamma}I-\mu_{I_{12}}(I_{12}- Q_ {I_{12}}) \label{sec11equ7}\\
        \frac{d I_{15}}{dt} &=& \beta_ {I_{15}} I_{\gamma}I-\mu_{I_{15}}(I_{15}- Q_ {I_{15}}) \label{sec11equ8}\\
        \frac{d I_{17}}{dt} &=& \beta_ {I_{17}} I_{\gamma}I-\mu_{I_{17}}(I_{17}- Q_ {I_{17}}) \label{sec11equ9}
\end{eqnarray}
\begin{table}[ht!]   
 \centering 
\begin{tabular}{||l| l|| } 
\hline\hline

\textbf{Symbols} &  \textbf{Biological Meaning}  \\  
\hline\hline 
$S$ & Susceptible schwann cells.\\
\hline

$I$ & Infected schwann cells.  \\
\hline
$B$ & Bacterial load.  \\
\hline
$I_{\gamma}$ & Concentration of IFN-$\gamma$.  \\
\hline
$T_{\alpha}$ & Concentration of TNF-$\alpha$.  \\
\hline
$I_{12}$ & Concentration of IL-12.  \\
\hline
$I_{15}$ & Concentration of IL-15.  \\
\hline
$I_{17}$ & Concentration of IL-17.  \\

\hline
$\omega$& Natural birth rate of the  susceptible cells. \\

\hline
$\beta$ & Rate at which schwann cells are infected. \\

\hline
$\gamma$& Death rate of the  susceptible cells due to cytokines.  \\

\hline
$\mu_{1}$ & Natural death rate of  schwann cells and  infected 
 schwann cells. \\

\hline
$\delta$ & Death rate of infected schwann cells  due to cytokines. \\

\hline
$\alpha$ & Burst rate of  infected schwann cells realising the bacteria. \\

\hline

$y$ & Rates at which {\textit{ M. Leprae }}
is removed by cytokines. \\

\hline
$\mu_{2}$ & Natural death rate of {\textit{ M. Leprae }}.  \\
\hline 
$\alpha_{I_\gamma}$ & Production rate of IFN-$\gamma$. \\

\hline
$\delta^{I_{\gamma}}_{T_{\alpha}}$ & Inhibition of IFN-$\gamma$ due to TNF-$\alpha$. \\

\hline

$\delta^{I_{\gamma}}_{I_{12}}$ & Inhibition of IFN-$\gamma$ due to IL-12.\\

\hline

$\delta^{I_{\gamma}}_{I_{15}}$ & Inhibition of IFN-$\gamma$ due to IL-15. \\

\hline

$\delta^{I_{\gamma}}_{I_{17}}$ & Inhibition of IFN-$\gamma$ due to IL-17. \\

\hline

$\mu_{I_{\gamma}}$ & Decay rate  of IFN-$\gamma$. \\

\hline

$\beta_{T_{\alpha}}$ & Production rate of TNF-$\alpha$. \\

\hline

$\mu_{T_{\alpha}}$ & Decay rate  of TNF-$\alpha$. \\

\hline

$\alpha_{I_{10}}$ & Production rate of $IL-10$. \\
\hline

$\delta^{I_{10}}_{I_{\gamma}}$ & Inhibition $IL-10$ of  due to  IFN-$\gamma$. \\
\hline

$\mu_{I_{10}}$ & Decay rate  of $IL-10$. \\

\hline

$\beta_{I_{12}}$ & Production rate of $IL-12$. \\

\hline

$\mu_{I_{12}}$ & Decay rate  of $IL-12$. \\

\hline

$\beta_{I_{15}}$ & Production rate of $IL-15$. \\

\hline

$\mu_{I_{15}}$ & Decay rate  of $IL-15$. \\

\hline

$\beta_{I_{17}}$ & Production rate of $IL-17$. \\

\hline

$\mu_{I_{17}}$ & Decay rate  of $IL-17$. \\

\hline
$Q_{I_{\gamma}}$ & Quantity of $IFN-\gamma$ before infection. \\

\hline

$Q_{T_{\alpha}}$ & Quantity of $TNF-\alpha$ before infection. \\
\hline

$Q_{I_{10}}$ & Quantity of $I_{10}$ before infection. \\
\hline

$Q_{I_{12}}$ & Quantity of $I_{12}$ before infection. \\

\hline

$Q_{I_{15}}$ & Quantity of $I_{15}$ before infection. \\

\hline

$Q_{I_{17}}$ & Quantity of $I_{17}$ before infection. \\

\hline\hline

\end{tabular}
\end{table} \vspace{.2cm}

     \hspace{0.5 in} Now mathematically  we define the set of all  control variables as follows:
     
     $$U=\Big\{D_{ij}(t),D_{ij}(t)\in[0,D_{ij}max],  1\leq i,j \leq 3,\ ij \neq 32,  t\in[0,T]\Big\}$$
     
    \hspace{0.2 in} Here $D_{ij}max$ represents the upper limit  of the corresponding  control variable which depends on the availability and limit of the drugs recommended for patients   and $T$ is the final time of observation. \\
     
     \hspace{0.2 in}  Since the drugs in MDT can lead to some hazards,  we consider a cost functional that minimizes the drug concentrations along with the infected cell count and bacterial load. \\
     
     Based on this we define  the following cost function:
     
\begin{equation}
  \begin{aligned}
   \mathcal{J}_{min}\big(I,B,D_{1},D_{2},D_{3}\big) &= \int_{0}^{T} \Big(I(t) + B(t)+P\big[D^{2}_{11}(t)+D^{2}_{12}(t)+D^{3}_{13}(t)\big]\\
   & +Q\big[D^{2}_{21}(t)+D^{2}_{22}(t)+D^{3}_{23}(t)\big] +  R\big[D_{31}^{2}(t)+D_{33}^{2}(t)\big]  \Big) dt 
  \label{opti}
  \end{aligned}
\end{equation} \\

Here the integrand of the cost function $\mathcal{J}_{min}$ is denoted by \\

\begin{equation}
  \begin{aligned}
    L\big(I,B, D_1,D_2,D_3\big) &= \Big(I(t) + B(t)+P[D^{2}_{11}(t)+D^{2}_{12}(t)+D^{3}_{13}(t)]\\
   & +Q[D^{2}_{21}(t)+D^{2}_{22}(t)+D^{3}_{23}(t)]
   + R\big[D_{31}^{2}(t)+D_{33}^{2}(t)\big] \Big) 
  \end{aligned}
\end{equation}
which denotes the running cost and is commonly known as \textit{Lagrangian} of the optimal control problem. \\

Now the admissible set of solution of the above optimal control problem  (\ref{sec11equ1}) - (\ref{opti}) is given by 

$$\Omega =\Big\{\big(I,B,D_1,D_2,D_3\big): I, B \hspace{0.1in} satisfying  \hspace{0.1in} (\ref{sec11equ1})-(\ref{sec11equ9})  \ \forall \ (D_1,D_2,D_3) \in U \Big\}$$\\
where $D_{1}=\big(D_{11},D_{12},D_{13} \big)$, $D_{2}=\big(D_{21},D_{22},D_{23} \big)$, $D_{3}=\big(D_{31},D_{33} \big)$. \\\\

 We next validate the model using 2D - Heat plots with  control variables as zero.

\section{Model validation through  2D Heat Plots} 
\hspace{0.2 cm} We validate the above framed  model based on the  clinical characteristics of leprosy. From the clinical studies it can be seen that  the doubling rate of of the Bacteria (\textit{M.Leprae}) is approximately $14$ days \cite{levy2006mouse}.  We use this clinical characteristic to validate our model.\\
 
\hspace{0.2 cm} To generate these heat plots, we use two pairs of parameters and their rage of variation. We consider the pairs of parameters $\alpha-\gamma$  and $\alpha-y.$  The range of values for $\alpha$ was from $ 0.0563 \ to \ 0.0763, $ the range for $\gamma$ is from $ 0.15 \ to \ 0.2090 $ and the range for   $y$ is from $ 0.0002 \ to \ 0.5003. $  All the other parameter values are taken from the table \ref{Table:1} and a feasible initial condition  $S(0)=5200$, $I(0)=0$ , $B(0)=40$, $I_{\gamma}(0)=5 $, $T_{\alpha}(0)=5$, $I_{10}(0)=15 $, $I_{12}(0)=12 $, $I_{15}(0)=12 $ and $I_{17}(0)=10$ was chosen. Using the above values we simulate the model through MATLAB taking $50$ pairs of different values for each  of the parameters within the rage provided and we consider the value of the control variables as zero at that time. Then we record the value of $B(t)$ at $14^{th}$ day  in a $50 \cross 50 $ matrix. We used the function \texttt{imagesc()} to create these heat plots in MATLAB.\\

\begin{figure} 
    \begin{subfigure}{0.45\textwidth}
        \includegraphics[width=\textwidth]{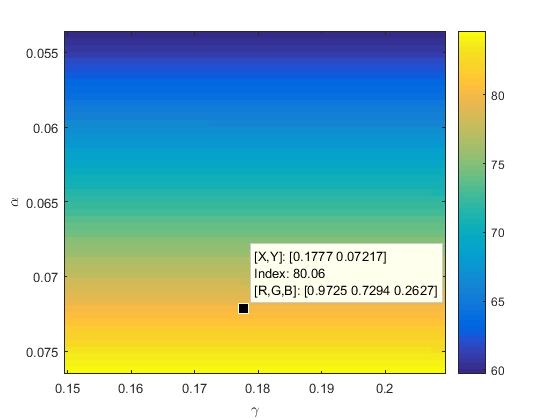}
        \caption{}
        \label{2a}
    \end{subfigure}
    \begin{subfigure}{0.45\textwidth}
        \includegraphics[width=\textwidth]{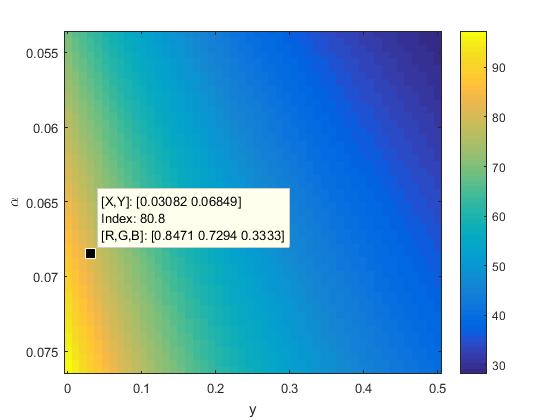}
        \caption{}
        \label{2b}
    \end{subfigure}
    
    \caption{Heat Plots (a)Taking $50 \cross 50$ pairs of values of the parameters $\alpha$ and $\gamma$, \\ 
            (b)Taking $50 \cross 50$ pairs of values of the parameters $\alpha$ and $y$  }
    \label{fig:2}
\end{figure}

In the figure \ref{2a}, the value of $\gamma$ is taken on ordinate and $\alpha$  on abscissa.  From the color bar of this figure, we see that the model is able to reproduce the characteristic that the initial  count of bacteria doubles   in $14$ days (in this case it $80$ with $B(0)=40$). Same inference can be made for figure \ref{2b}  where $y$ is taken on ordinate and $\alpha$ is taken on abscissa.

\section{Existence of an Optimal Solution}

 \hspace{0.2 in} In this section we establish the existence of solution for the optimal control problem (\ref{sec11equ1}) -  (\ref{opti}) using the existence theorem $2.2$ of \cite{boyarsky1976existence}.
 \begin{theorem}
 There exists an 8- tuple of optimal controls ($D_{1}^*(t),D_{2}^*(t),D_{3}^*(t)$) in the set of admissible controls $U$ and hence the optimal state variables,($S^*(t),I^*(t),B^*(t)$) such that the cost function is minimized i.e.
$$\mathcal{J}_{min}\big(I^*(t),B^*(t),D_{1}^*(t),D_{2}^*(t),D_{3}^*(t )\big)=\min_{D_1,D_2,D_3\in U} \mathcal{J}_{min}\big(I,B,D_{1},D_{2},D_{3}\big)$$
corresponding to the control system(\ref{sec11equ1}) -  (\ref{opti}), where $D_{1}=\big(D_{11},D_{12},D_{13} \big)$, $D_{2}=\big(D_{21},D_{22},D_{23} \big)$, $D_{3}=\big(D_{31},D_{33} \big)$ . 
\end{theorem}
\begin{proof}
We consider $\frac{ds}{dt}=f^{1}(t,x,D)$, $\frac{dI}{dt}=f^{2}(t,x,D)$, $\frac{dB}{dt}=f^{3}(t,x,D)$, $\frac{dI_{\gamma}}{dt}=f^{4}(t,x,D)$, $\frac{dT_{\alpha}}{dt}=f^{5}(t,x,D)$, $\frac{dI_{10}}{dt}=f^{6}(t,x,D)$, $\frac{dI_{12}}{dt}=f^{7}(t,x,D)$, $\frac{dI_{15}}{dt}=f^{8}(t,x,D)$ and $\frac{dI_{17}}{dt}=f^{9}(t,x,D)$ of the control system (\ref{sec11equ1}) -  (\ref{opti}). Here $x\in X$ denotes the state variables and $D\in U$ denotes the control variables. With $f=(f^1,f^2,f^3,f^4,f^5,f^6,f^7,f^8,f^9)$ we see that  $X \subseteq \mathbb{R}^{9}$ and,
$$f:[0,T]\times X \times U \to \mathbb{R}^{9}$$
is a continuous function of $t$ and $x$ for each  $ D_{ij} \in U$. \\

Now we intend to show that $(F1)$ to $(F3)$ of theorem 2.2 of \cite{boyarsky1976existence} holds true for all $f^i$'s. \\

 \textbf{F1:} Here each of the $f^i$'s has a continuous and bounded partial derivative implying that  $f$ is  Lipschitz's continuous.\\
 
 \textbf{F2:} We  consider $g_1(D_{11},D_{21},D_{31}) =-D_{11}-D_{12}+D_{31}, $  which is bounded on $U$.
 
 Thus 
 \begin{equation*}
  \begin{aligned}
  \frac{ f^1(t,x,D^{(1)})- f^1(t,x,D^{(2)})}{\big[g_1(D^{(1)})-g_1(D^{(2)})\big]}
 &= \frac{\big[D^{(2)}_{11}+D^{(2)}_{12}-D^{(2)}_{31}-D^{(1)}_{11}-D^{(1)}_{12}+D^{(1)}_{31}\big]S}{\big[D^{(2)}_{11}+D^{(2)}_{12}-D^{(2)}_{31}-D^{(1)}_{11}-D^{(1)}_{12}+D^{(1)}_{31}\big]}\\
   & \leq \eta S = F_1(t,x)\\
 \therefore  f^1(t,x,D^{(1)})- f^1(t,x,D^{(2)}) &\leq F_1(t,x)\bullet \big[g_1(D^{(1)})-g_1(D^{(2)})\big]
  \end{aligned}
\end{equation*}
 Here $\eta > 1$ is a real number. Moreover since $U$ compact and $g_1$ is continuous,  we have $g_{1}(U)$ to be compact. Also since the function  $g_{1}(U)$ is linear so the range of $g_{1}$ i.e. $g_{1}(U)$ will be convex.  Since $U$ is non-negative so $g^{-1}_{1}$ is non-negative. \\

 Similarly for $f^{2}(t,x,D),$ we choose $g_2(D_{12},D_{22}) = -D_{12}-D_{22}$ and $F_2(t,x)= I$ and   \textbf{F2} can be proven in a similar way.\\

 Now for $f^{3}(t,x,D)$ we choose $g_3(D_{23},D_{33}) = -D^{2}_{23}-D_{33} $ 
 
Hence
 
\begin{equation*}
  \begin{aligned}
  \frac{ f^3(t,x,D^{(1)})- f^3(t,x,D^{(2)})}{\big[g_3(D^{(1)})-g_3(D^{(2)})\big]}
   &= \frac{\big[D^{2(2)}_{23}+D^{(2)}_{33}-D^{2(1)}_{23}-D^{(1)}_{33}\big]I-[D_{13}^{2(1)}-D_{13}^{2(2)}]B}{\big[D^{2(2)}_{23}+D^{(2)}_{33}-D^{2(1)}_{23}-D^{(1)}_{33}\big]}\\
   & \leq \frac{\big[D^{2(2)}_{23}+D^{(2)}_{33}-D^{2(1)}_{23}-D^{(1)}_{33}\big]I}{\big[D^{2(2)}_{23}+D^{(2)}_{33}-D^{2(1)}_{23}-D^{(1)}_{33}\big]}=I=F_3(t,x)  \ \bigg(\text{provided} \ D^{2(1)}_{13} \geq D^{2(2)}_{13} \bigg)\\
 \therefore  f^3(t,x,D^{(1)})- f^3(t,x,D^{(2)}) &\leq F_3(t,x)\cdot \big[g_3(D^{(1)})-g_3(D^{(2)})\big]
  \end{aligned}
\end{equation*}
 But to satisfy this hypothesis for remaining $f^i$'s we take the help of corollary 2.1 of \cite{boyarsky1976existence} i.e. \textbf{F4} must be satisfied in the place of \textbf{F2}. Hence considering $g_{4}(D_1,D_2,D_3)=0$, which is bounded measurable function and $F_4(t,x)=1$ we establish the relation
\begin{equation*}
       f^3(t,x,D^{(1)})- f^3(t,x,D^{(2)})=0 \leq 1= 1\cdot 0= F_4(t,x) \cdot \big[g_4(D^{(1)}-D^{(2)})\big]
\end{equation*}
 Similarly taking $F_i(t,x)=i$ and $g_{i}(D_1,D_2,D_3)=0$ for $i=5,6,7,8,9$ we get  the relations 
\begin{equation*}
    f^i(t,x,D^{(1)})- f^i(t,x,D^{(2)}) \leq F_i(t,x) \cdot \big[g_i(D^{(1)}-D^{(2)})\big]
\end{equation*}
 Hence we are done in satisfying the hypothesis \textbf{F2}\\

 \textbf{F3:}  Since $S,I,B$ and $f_{i}(x,t)=1$ are bounded on $[0,T]$  so $F_{i}(\bullet,x^{u}) \in \mathcal{L}_{1}$\\

 Now we have to show that the running cost function 
 $$C(t,x,D)=I(t) + B(t)+P\big[D^{2}_{11}(t)+D^{2}_{12}(t)+D^{3}_{13}(t)\big]
   +Q\big[D^{2}_{21}(t)+D^{2}_{22}(t)+D^{3}_{23}(t)\big] + R\big[|D_{3}(t)|^{2}\big]$$
satisfies the  conditions $(C1)-(C5)$ of  \textit{Theorem 2.2} of \cite{boyarsky1976existence}.

Here $C:[0,T]\times X \times U \to \mathbb{R}$\\

\textbf{C1:} We see that $C(t,\bullet,\bullet)$ is a continuous function  as it is sum of continuous functions which are functions of $t\in [0,T]$.\\

\textbf{C2:}  $S, I$ and $B$  and all $D_{ij}$'s are bounded implying that $C(\bullet,x,D)$ is bounded and hence measurable for each $x\in X$ and $D_{ij} \in U$.\\

\textbf{C3:} Consider $\Psi(t) = \kappa$ such that $ \kappa = \min \{I(0),B(0)\} $ then $\Psi$ will bounded such that for all $t\in [0,T]$, $x \in X$ and $D_{ij}\in U,$ we have 
$$C(t,x,D)\geq \Psi(t)$$\\

\textbf{C4:} Since $C(t,x,D)$ is sum of the function  which are convex in $U$ for each fixed $(t,x)\in [0,T]\times X $  therefore $C(t,x,D)$ follows the same.\\

\textbf{C5:} Using similar type of argument, we can easily show that for each fixed $(t,x)\in [0,T]\times X $, $C(t,x,D)$ is a monotonically increasing function. \\

Hence we have shown that the optimal control problem satisfies the all hypothesis of the \textit{Theorem 2.2} of \cite{boyarsky1976existence}.
Therefore there exists a $8$- tuple of optimal controls  $\big(D_{1}^*(t),D_{2}^*(t),D_{3}^*(t)\big)$ in the set of admissible controls $U$ such that the cost function is minimized.

\end{proof}

\section{Numerical Studies for the Optimal Control Problem}
\subsection{Theory}
Here we discuss the technique to evaluate the aforementioned
 optimal control problem. To evaluate the optimal control variables and the optimal state variables, we use the \textit{Forward Backward Sweep Method} \cite{mcasey2012convergence}  and  Pontryagin Maximum Principle \cite{liberzon2011calculus}. \\

 The Hamiltonian of the control system (\ref{sec11equ1}) - (\ref{opti}) is given by
 
\begin{equation}
\begin{aligned}
    \mathcal{H}\big(I,V, D_1,D_2,D_3,\lambda\big) = & I(t) + B(t)+P\big[D^{2}_{11}(t)+D^{2}_{12}(t)+D^{3}_{13}(t)\big]
 +Q\big[D^{2}_{21}(t)+D^{2}_{22}(t)+D^{3}_{23}(t)\big] +\\ &  R\big[D_{31}^{2}(t)+D_{33}^{2}(t)\big]+ \lambda_{1}\frac{dS}{dt} + \lambda_{2}\frac{dI}{dt} + \lambda_{3}\frac{dB}{dt} + \lambda_{4}\frac{dI_{\gamma}}{dt} + \lambda_{5}\frac{dT_{\alpha}}{dt}+\lambda_{6}\frac{dI_{10}}{dt}\\
 & +\lambda_{7}\frac{dI_{12}}{dt} +\lambda_{8}\frac{dI_{15}}{dt}
  +\lambda_{9}\frac{dI_{17   }}{dt}
\end{aligned}
\end{equation} 
where $\lambda = (\lambda_1,\lambda_2,\lambda_3,\lambda_4,\lambda_5,\lambda_6,\lambda_7,\lambda_8,\lambda_9)$ is the as co-state variable/adjoint vector. \\

Now using \textit{Pontryagin Maximum Principle} with $D^*=(D_1^*,D_2^*,D_3^*)$ and $X^*=(x_1,x_2,\dots,x_9)$ being the optimal control and state variable respectively,  there exits a optimal co-state variable for which
\begin{equation}
    \frac{dx_i}{dt}=\frac{\partial \mathcal{H}(X^*,D^*,\lambda^*)}{\partial \lambda_i }
\label{steq}    
\end{equation}
\begin{equation}
    \frac{d \lambda_i}{dt}=-\frac{\partial \mathcal{H}(X^*,D^*,\lambda^*)}{\partial x_i}
    \label{adeq}
\end{equation}
Clearly the  (\ref{steq}) will be equivalent to the control system (\ref{sec11equ1})- (\ref{opti}) and system of ODE's for co-state variables should satisfy the following.
\begin{equation}
    \begin{aligned}
         \frac{d\lambda_1}{dt} &= \big(\beta B +\mu_1+\gamma+D_{11}+D_{21}-D_{31}\big)\lambda_{1}-\big(\beta B\big) \lambda_2\\
         \frac{d\lambda_2}{dt} &= \big(\mu_1+\delta+D_{12}+D_{22}\big)\lambda_2-\big(\alpha-D_{23}^2-D_{33}\big)\lambda_{3}-\alpha_{I_{\gamma}}\lambda_4-\beta_{T_{\alpha}}I_{\gamma}\lambda_{5}\\
         & -\alpha_{I_{10}}\lambda_6-\beta _{I_{12}}I_{\gamma}\lambda_7-\beta_{I_{15}}I_{\gamma}\lambda_8-\beta_{I_{17}}I_{\gamma}\lambda_9-1 \\
         \frac{d\lambda_3}{dt} &= \big(\beta S)\lambda_{1}-\big(\beta S)\lambda_{2}+\big(y+\mu_{2}+D_{13}^{2}\big)\lambda_{3}-1\\
         \frac{d\lambda_4}{dt} &= \mu_{I_{\gamma}}I\lambda_{4}-\beta_{T_{\alpha}}I\lambda_5+\delta^{I_{10}}_{I_{\gamma}}\lambda_6-\beta_{I_{12}}I\lambda_7- \beta_{I_{15}}I\lambda_8-\beta_{I_{17}}I\lambda_9\\
         \frac{d\lambda_5}{dt} &= \delta^{I_{\gamma}}_{T_{\alpha}}I\lambda_4+\mu_{T_{\alpha}}\lambda_5\\
         \frac{d\lambda_6}{dt} &= \mu_{I_{10}}\lambda_6\\
         \frac{d\lambda_7}{dt} &= \delta^{I_{\gamma}}_{I_{12}}I\lambda_4+\mu_{I_{12}}\lambda_7\\
         \frac{d\lambda_8}{dt} &= \delta^{I_{\gamma}}_{I_{15}}I\lambda_4+\mu_{I_{15}}\lambda_8\\
         \frac{d\lambda_9}{dt} &= \delta^{I_{\gamma}}_{I_{17}}I\lambda_4+\mu_{I_{17}}\lambda_9
    \end{aligned}
\end{equation}
and the  transversality condition $\lambda_{i}(T)=\eval{\frac{\partial \phi}{\partial t}}_{t=T}=0$ for all $i=1,2,\dots,9$ (where $\phi$ is the final cost function and here $\phi \equiv 0$). \\

Now to obtain the optimal value of the  controls  we use the \textit{Newton's Gradient method} for optimal control problem \cite{edge1976function}. For this a recursive formula is being used to update the control in each step of numerical simulation i.e.
\begin{equation}
   D_{ij}^{k+1}(t)=D_{ij}^{k}(t)+\theta_k d_k
   \label{ctrl1}
\end{equation}

where $D_{ij}^{k}(t)$ is the value of the control at $k^{th}$ iteration  at time instance $t$, $d_k$ is the direction  and $\theta$ is the step size. Usually direction in \textit{Newton's Gradient method} is evaluated by the negative of the gradient of the objective function i.e. $d_k=-g_{ij}(D_{ij}^{k})$ and here we take $g_{ij}(D_{ij}^{k})=\eval{\frac{\partial\mathcal{H}}{\partial D_{ij}}}_{D_{ij}^{k}(t)}$ as described in \cite{edge1976function}.  The step size $\theta$ is evaluated at each iteration by linear search technique that minimizes  the Hamiltonian, $\mathcal{H}$. Therefor the previous formula (\ref{ctrl1}) will be as:

\begin{equation}
   D_{ij}^{k+1}(t)=D_{ij}^{k}(t)- \theta_k \eval{\frac{\partial\mathcal{H}}{\partial D_{ij}}}_{D_{ij}^{k}(t)}
   \label{ctrl2}
\end{equation}
Now to execute the idea above we have to calculate the gradient for each control i.e. $g_{ij}(D_{ij})$ which are as follows 
\begin{equation*}
    \begin{aligned}
        g_{11}(D_{11})&=2PD_{11}(t)-\lambda_{1}S(t)\\
        g_{12}(D_{12})&=2PD_{12}(t)-\lambda_{2}I(t)\\
        g_{13}(D_{13})&=3PD_{13}^2(t)-2\lambda_{3}D_{13}(t)B(t)\\
        g_{21}(D_{21})&=2QD_{21}(t)-\lambda_{1}S(t)\\
        g_{22}(D_{22})&=2QD_{22}(t)-\lambda_{2}I(t)\\
        g_{23}(D_{23})&=3QD_{23}^2(t)-2\lambda_{3}D_{23}(t)I(t)\\
        g_{31}(D_{31})&=2RD_{31}(t)+\lambda_{1}S(t)\\
        g_{33}(D_{33})&=2RD_{33}(t)-\lambda_{3}I(t)
    \end{aligned}
\end{equation*}

\subsection{Numerical Simulations}
\hspace{0.2 in} In this section we perform the numerical simulations to study the correlation of  cytokines levels in Type-1 Lepra reaction and the drugs involved in MDT in a qualitative manner.  \\
 
\hspace{0.2 in} The value of the parameters used are collected from various clinical papers and appropriate references are cited in the table \ref{Table:1}. For some of the parameters like $\mu,\gamma$ and $\delta$ the doubling time was available, hence they were estimated  using the formula
$$ rate\hspace{0.07 in}\%= \frac{log(2)}{doubling \hspace{0.07 in} time} \times 100\approx \frac{70}{doubling \hspace{0.07 in} time}$$

\hspace{0.2 in} We then divide these rate percentages by 100 to get values of these parameters. Some of the cases we have taken average of the result yield from different medium such as $7-AAD$, $TUNNEL$ as described in  \cite{oliveira2005cytokines}. Some parameters are finely tuned in order to satisfy certain hypothesis/assumptions for convenience of numerical simulation.  \\

For these simulations we consider the time duration of $100$ days i.e. 
($T = 100$) and the parameter values are chosen as  $\omega = 20.9$, $\beta = 0.3$, $\mu_1 = 0.00018$,
$\gamma = 0.01795$, $\delta = 0.2681$, $\alpha = 0.2$, $y = 0.3$, $\mu_2 = 0.57$, $\alpha_{I_\gamma}=0.0003$, $\delta^{I_{\gamma}}_{T_{\alpha}}=0.00554$, $\delta^{I_{\gamma}}_{I_{12}}=0.00903$, $\delta^{I_{\gamma}}_{I_{15}}=0.00625$, $\delta^{I_{\gamma}}_{I_{17}}=0.00499$, $\mu_{I_{\gamma}}=2.16$, $\beta_{T_{\alpha}}=0.004$, $\mu_{T_{\alpha}}=1.112$, $\alpha_{I_{10}}=0.044$, $\delta^{I_{10}}_{I_{\gamma}}=0.00146$, $\mu_{I_{10}}=16$, $\beta_{I_{12}}=0.011$, $\mu_{I_{12}}=1.88$, $\beta_{I_{15}}=0.025$, $\mu_{I_{15}}=2.16$, $\beta_{I_{17}}=0.029$, $\mu_{I_{17}}=2.34$, $Q_{I_{\gamma}}=0.1$, $Q_{T_{\alpha}}=0.14$ $Q_{I_{10}}=0.15$, $Q_{I_{12}}=1.11$, $Q_{I_{15}}=0.2$, $Q_{I_{10}}=0.317.$  \\

First
we have solved the system numerically without any drug intervention. All the numerical 
calculation were done in MATLAB  and we used $4^{th}$ order {Runge-Kutta} method to solve system of ODEs and to find the value of the $\theta$ in each iteration we used the \texttt{fminsearch()} function of MATLAB. Here we consider the initial value of the state variables as $S(0)=520$, $I(0)=275$ , $B(0)=250$, $I_{\gamma}(0)=50 $, $T_{\alpha}(0)=50$,$I_{10}(0)=75 $, $I_{12}(0)=125 $, $I_{15}(0)=125 $ and $I_{17}(0)=100$ as in \cite{ghosh2021mathematical,liao2013role}.  \\

 \begin{table}[ht!]   
 \centering 
\begin{tabular}{||c|c|c||} 
\hline\hline

\textbf{Symbols} &  \textbf{Values} & \textbf{Units} \\  

\hline \hline
$\omega$ & 0.0220\cite{kim2017schwann}& $pg.ml^{-1}.day^{-1} $ \\

\hline
$\beta$ & 3.4400\cite{jin2017formation} & $pg.ml^{-1}.day^{-1} $  \\

\hline
$\gamma$& 0.1795 \cite{oliveira2005cytokines}  & $day^{-1}$ \\

\hline
$\mu_{1}$ & 0.0018 \cite{oliveira2005cytokines} & $day^{-1}$  \\

\hline
$\delta$ & 0.2681 \cite{oliveira2005cytokines} & $day^{-1}$ \\

\hline
$\alpha$ & 0.0630\cite{levy2006mouse} & $pg.ml^{-1}.day^{-1}$ \\

\hline
$y$ & $0.0003$\cite{ghosh2021mathematical} & $day^{-1}$ \\
\hline
$\mu_{2}$ & 0.5700\cite{international2020international}& $day^{-1}$ \\
\hline
$\alpha_{I_\gamma}$ & 0.0003\cite{pagalay2014mathematical}& $pg.ml^{-1}.day^{-1} $ \\
\hline
$\delta^{I_{\gamma}}_{T_{\alpha}}$ &0.005540* & $pg.ml^{-1}$  \\

\hline

$\delta^{I_{\gamma}}_{I_{12}}$ & 0.009030* & $pg.ml^{-1}$ \\

\hline

$\delta^{I_{\gamma}}_{I_{15}}$ & 0.006250* & $pg.ml^{-1}$ \\

\hline

$\delta^{I_{\gamma}}_{I_{17}}$ & 0.004990*& $pg.ml^{-1}$ \\

\hline

$\mu_{I_{\gamma}}$ & 2.1600\cite{pagalay2014mathematical} & $day^{-1} $\\
\hline

$\beta_{T_{\alpha}}$ & 0.0040\cite{pagalay2014mathematical}&  $pg.ml^{-1}.day^{-1} $ \\
\hline
$\mu_{T_{\alpha}}$ & 1.1120\cite{pagalay2014mathematical}  & $day^{-1}$  \\
\hline
$\alpha_{I_{10}}$ & 0.0440\cite{liao2013role} & $pg.ml^{-1}.day^{-1} $  \\
\hline

$\delta^{I_{10}}_{I_{\gamma}}$ & 0.001460*& $pg.ml^{-1}$\\
\hline

$\mu_{I_{10}}$ & 16.000\cite{liao2013role} & $day^{-1}$ \\
\hline

$\beta_{I_{12}}$ &0.0110\cite{liao2013role}  & $pg.ml^{-1}.day^{-1}$\\
\hline
$\mu_{I_{12}}$ & 1.8800\cite{pagalay2014mathematical} & $day^{-1}$ \\
\hline
$\beta_{I_{15}}$ &0.0250\cite{su2009mathematical} & $pg.ml^{-1}.day^{-1} $   \\
\hline
$\mu_{I_{15}}$ & 2.1600\cite{su2009mathematical} & $day^{-1}$\\
\hline
$\beta_{I_{17}}$ & 0.0290\cite{su2009mathematical} & $pg.ml^{-1}.day^{-1} $  \\
\hline
$\mu_{I_{17}}$ & 2.3400\cite{su2009mathematical}& $day^{-1}$\\
\hline
$Q_{I_{\gamma}}$ & 0.1000\cite{talaei2021mathematical} & Relative concentration\\
\hline
$Q_{T_{\alpha}}$ & 0.1400\cite{brady2016personalized}  & Relative concentration \\
\hline 
$Q_{I_{10}}$ & 0.1500\cite{talaei2021mathematical} &  Relative concentration\\
\hline 
$Q_{I_{12}}$ &1.1100\cite{brady2016personalized}  & Relative concentration \\
\hline 
$Q_{I_{15}}$ & 0.2000\cite{brady2016personalized} & Relative concentration  \\
\hline 
$Q_{I_{17}}$ & 0.3170\cite{brady2016personalized} & Relative concentration \\
\hline \hline
\end{tabular}
\caption{Values of the parameters complied from clinical literature.The (*) marked values of the parameters are assumed.}
\label{Table:1}
\end{table}

\hspace{0.2 in} Further to simulate the system with controls, we use the
{Forward-backward sweep} method starting with the initial value of the controls as
zero and estimate the sate variables forward in time. Since the the transversality
conditions have the value of adjoint vector at end time $T,$ so the adjoint vector was calculated backward in time. \\

 Using the value of state variables and adjoint vector we
 calculate the  control variables at each time instance that get
updated in each iteration.The strategy to update controls is followed  by implementing the \textit{Newton's gradient method}  as expressed in equation (\ref{ctrl2}). We continue this till the convergence criterion is met as in
\cite{edge1976function}. \\

The weights  $P, Q$ and
$R$  in the cost function $\mathcal{J}_{min}$ are chosen based on their  \textit{hazard ratio} of the corresponding drugs. We chose the weights directly proportional to the hazard ratios.  In {Table} \ref{HR} the {hazard ratios} of the
different drugs are enlisted. We have chosen the weights $(P, Q$
and $S)$ proportional to  the hazard ratios i.e. $P = 1$, $Q = 1.99$ and $R = 7.1$.

\begin{table}[ht!]   
\centering 
\begin{tabular}{|c|c|c|} 
\hline

\textbf{Drugs} &  \textbf{Hazard Ratio } & \textbf{Source} \\ 

\hline\hline
Rifampin & 0.26 & \cite{bakker2005prevention} \\

\hline\hline
Dapsone & 0.99 & \cite{cerqueira2021influence}  \\

\hline\hline
Clofazimine & 1.85 & \cite{cerqueira2021influence}\\

\hline

\end{tabular}
\caption{Hazard Ratio of the drugs }
\label{HR}
\end{table}

We now numerically simulate the $S, I $ and $B$ populations  and  the cytokine levels with single control intervention, with two control interventions and finally with three control interventions of MDT. In each of these plots we also depict the no control intervention case for comparison purpose. In the next part we will discuss the findings of these simulations.

\subsection{Findings}

 \hspace{0.2 in} In each of the figures discussed in this section the first three frames (1-3) depicts the  dynamics of the respective compartments in the model (\ref{sec11equ1}) -  (\ref{sec11equ9})  with individual  drug administration, a combination of two drugs administration  and all the three drugs in MDT administration, respectively as control variables/interventions. The further frames below are the magnified versions of either Frame-1 and Frame-2 or Frame-1, Frame-2 and Frame-3 and are depicted for better clarity purpose to the reader.

\begin{figure}[ht!]
    \centering
    \includegraphics[height = 6cm, width =12cm]{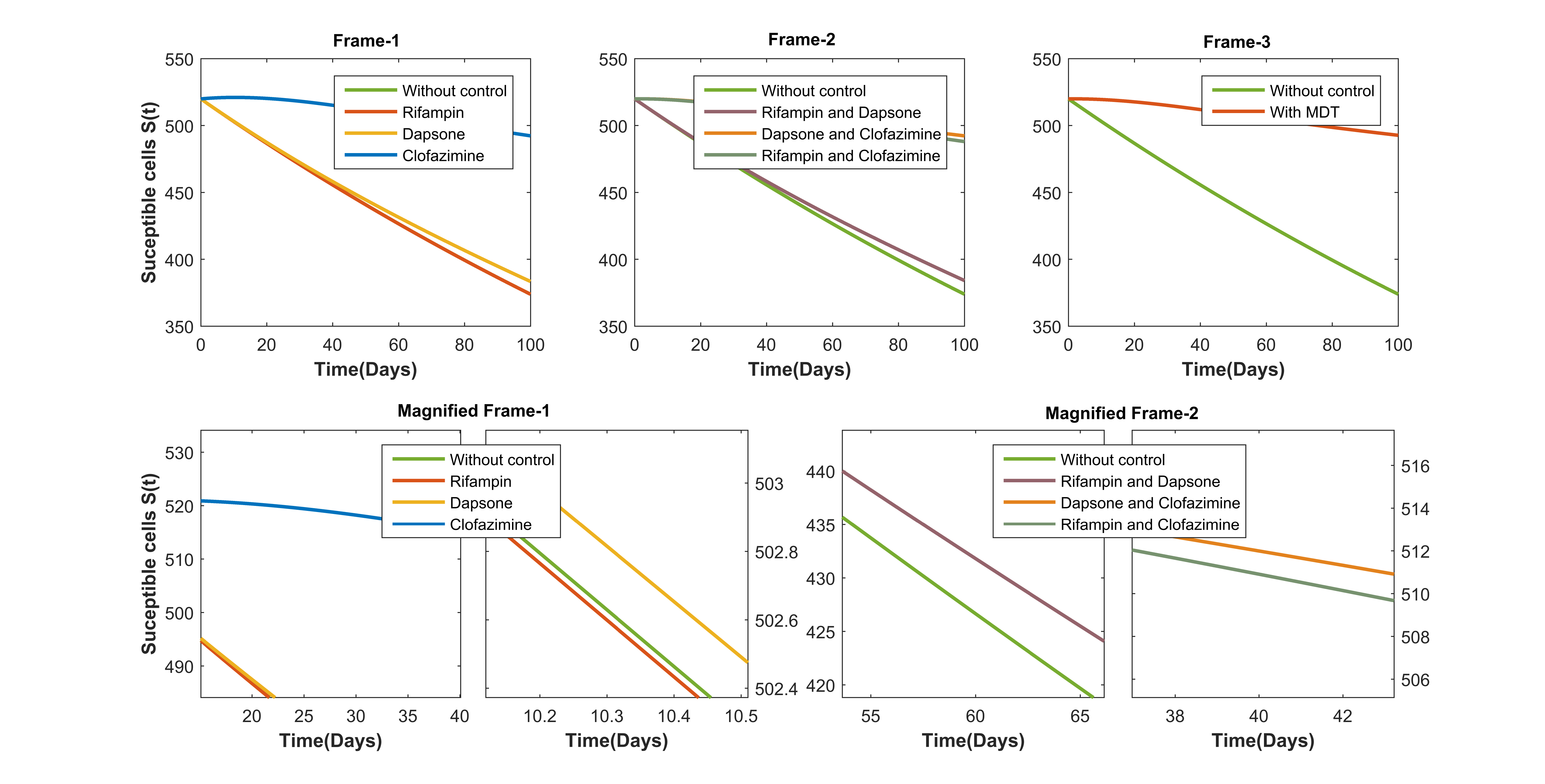}
    \caption{Plot depicting the dynamics of the susceptible cells $ S(t) $ on administration of MDT drugs individually (Frame-1), in combination of two (Frame-2) and all the three MDT drugs together (Frame-3). Magnified versions of Frame-1 and Frame-2 are also depicted for better clarity purpose to the reader. }
    \label{SC}
\end{figure} 

 Figure \ref{SC} depicts the  dynamics of the susceptible cells $ S(t) $   with individual  drug administration, a combination of two drugs administration  and all the three drugs in MDT administration, respectively.  From Frame-1 and its  magnification it can be seen that the drug \textit{clofazimine} and \textit{dapsone} has positive impact on susceptible cell count i.e these drugs increases the count of susceptible cell and among them \textit{clofazimine} acts most effectively.  On the other hand \textit{rifampin} decreases the count of the susceptible cells. In case  of combination of two drugs, \textit{rifampin} and \textit{dapsone} decreases the number of susceptible cells. But other two combinations increases the count and among them combination of \textit{clofazimine} and \textit{dapsone} have more impact in increment of the cell count.  Finally the Frame-3 depicts that the combination of three drugs increases the susceptible cell count. \\
\begin{figure}[ht!]
    \centering
    \includegraphics[height = 6cm, width =12cm]{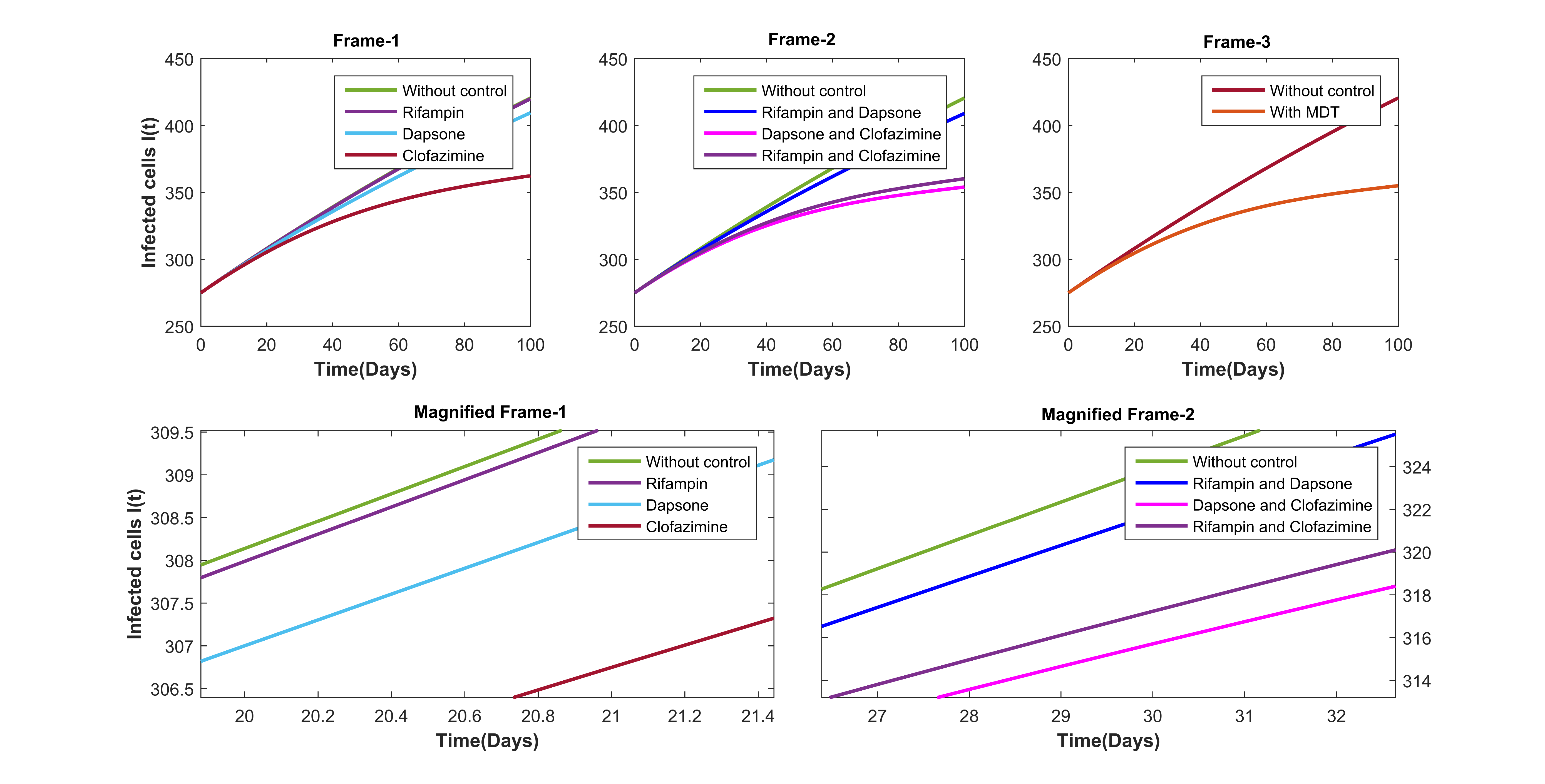}
    \caption{Plot depicting the dynamics of the infected cells $ I(t) $ on administration of MDT drugs individually (Frame-1), in combination of two (Frame-2) and all the three MDT drugs together (Frame-3). Magnified versions of Frame-1 and Frame-2 are also depicted for better clarity purpose to the reader.}
    \label{IC}
\end{figure} 

 Figure \ref{IC} depicts the dynamics of the infected cells $ I(t) $ on administration of MDT drugs. Analyzing the magnified Frame-1 we see that each of the MDT drugs is  helpful in reducing the infected cells.  In the context of   effectiveness we see that  the  \textit{clofazimine} takes the top position followed by \textit{dapsone} and followed by \textit{rifampin}. From the magnified  Frame-2 we see that among the combination of two drugs the combination consisting of \textit{clofazimine} and \textit{dapsone} acts most effectively where as the combination of \textit{rifampin} and \textit{dapsone} has the least impact.  All the three drugs  of  MDT when administered in combination  reduces the infected cell count  the best.\\

\begin{figure}[ht!]
    \centering
    \includegraphics[height = 6cm, width =12cm]{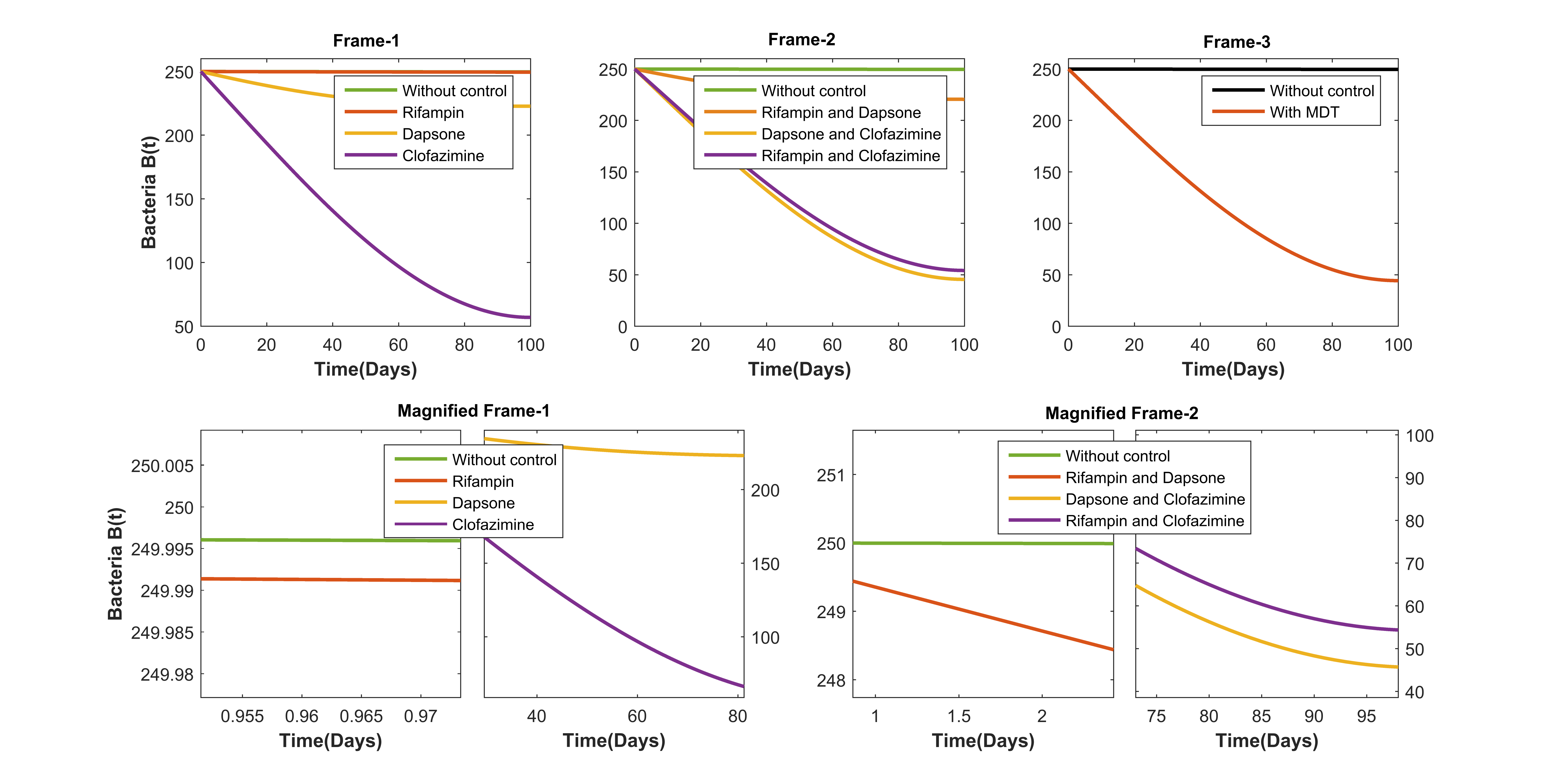}
    \caption{Plot depicting the dynamics of the bacterial load $ B(t) $ on administration of MDT drugs individually (Frame-1), in combination of two (Frame-2) and all the three MDT drugs together (Frame-3). Magnified versions of Frame-1 and Frame-2 are also depicted for better clarity purpose to the reader.}
    \label{BL}
\end{figure} 
From the maginifications in  figure \ref{BL} we see that \textit{clofazimine} is the most effective and \textit{rifampin} is the least effective drug in reducing the bacterial load when drugs are administered individually. In case administration of combination of two drugs,   \textit{dapsone} and \textit{clofazimine} combination has the most impact and \textit{rifampin} and \textit{dapsone} has the least impact. All the three drugs  of  MDT when administered in combination  reduces the bacterial load  the best.\\\
\begin{figure}[ht!]
    \centering
    \includegraphics[height = 6cm, width =12cm]{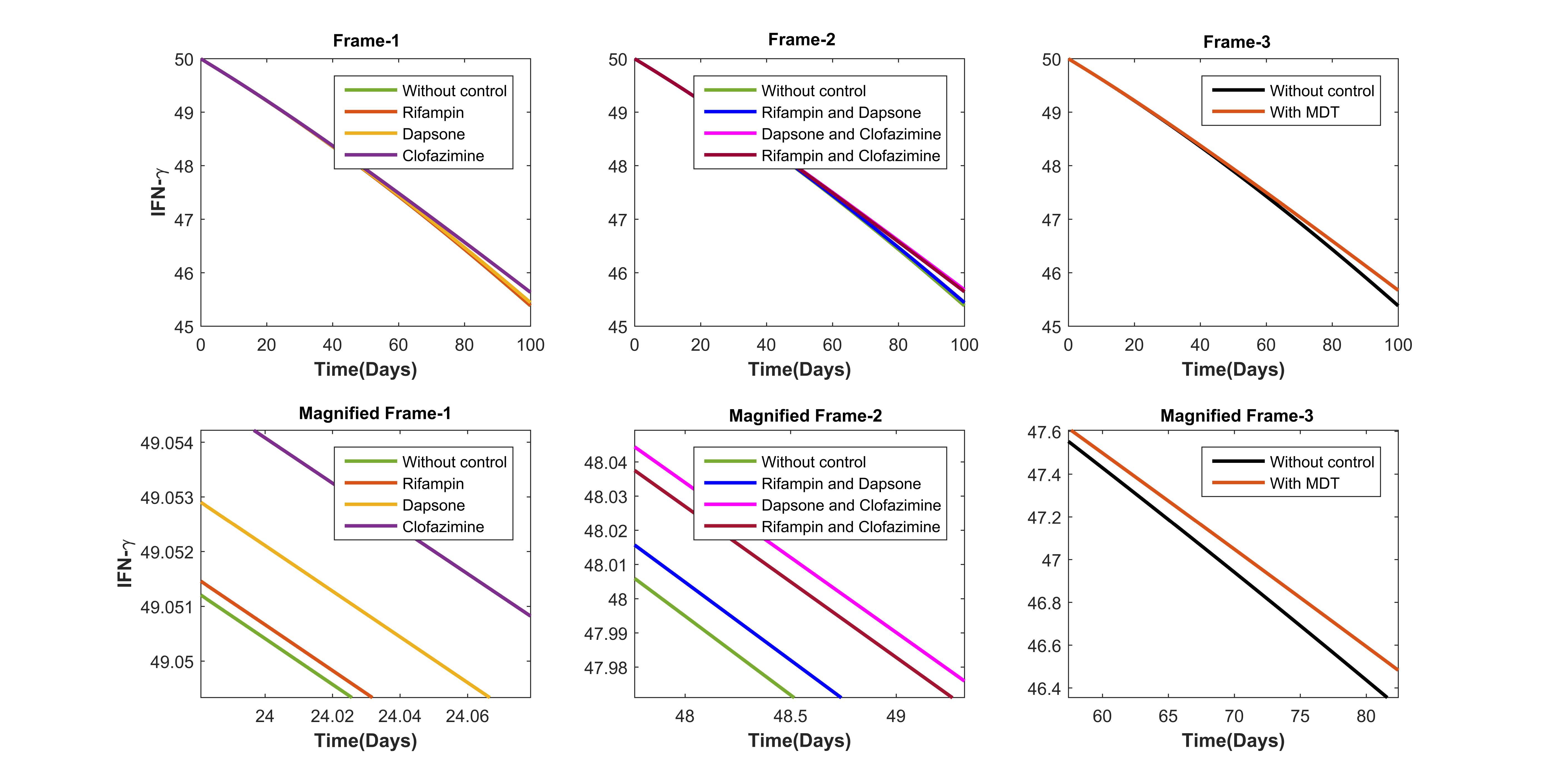}
    \caption{Plot depicting the dynamics of the $IFN-\gamma$ on administration of MDT drugs individually (Frame-1), in combination of two (Frame-2) and all the three MDT drugs together (Frame-3). Magnified versions of Frame-1, Frame-2 and Frame-3 are also depicted for better clarity purpose to the reader.}
    \label{IG}
\end{figure}

The figure \ref{IG} provides us the most important information that with out any  intervention of drugs, the level of $IFN-\gamma$ goes on decreasing during \textit{Lepra reaction} and upon administration of drugs in MDT  the levels of the $IFN-\gamma$ get enhanced. In the case of administration of  the drugs individually  we see that \textit{clofazimine} enhances the levels of $IFN-\gamma$ the highest followed by \textit{dapsone} and further followed by \textit{rifampin}. In case of combination of two drugs administration we see that \textit{dapsone} and \textit{clofazimine} enhances the levels of $IFN-\gamma$ the highest followed by \textit{rifampin} and \textit{clofazimine} and further followed by \textit{rifampin} and \textit{dapsone}. \\

\begin{figure}[ht!]
    \centering
    \includegraphics[height = 6cm, width =12cm]{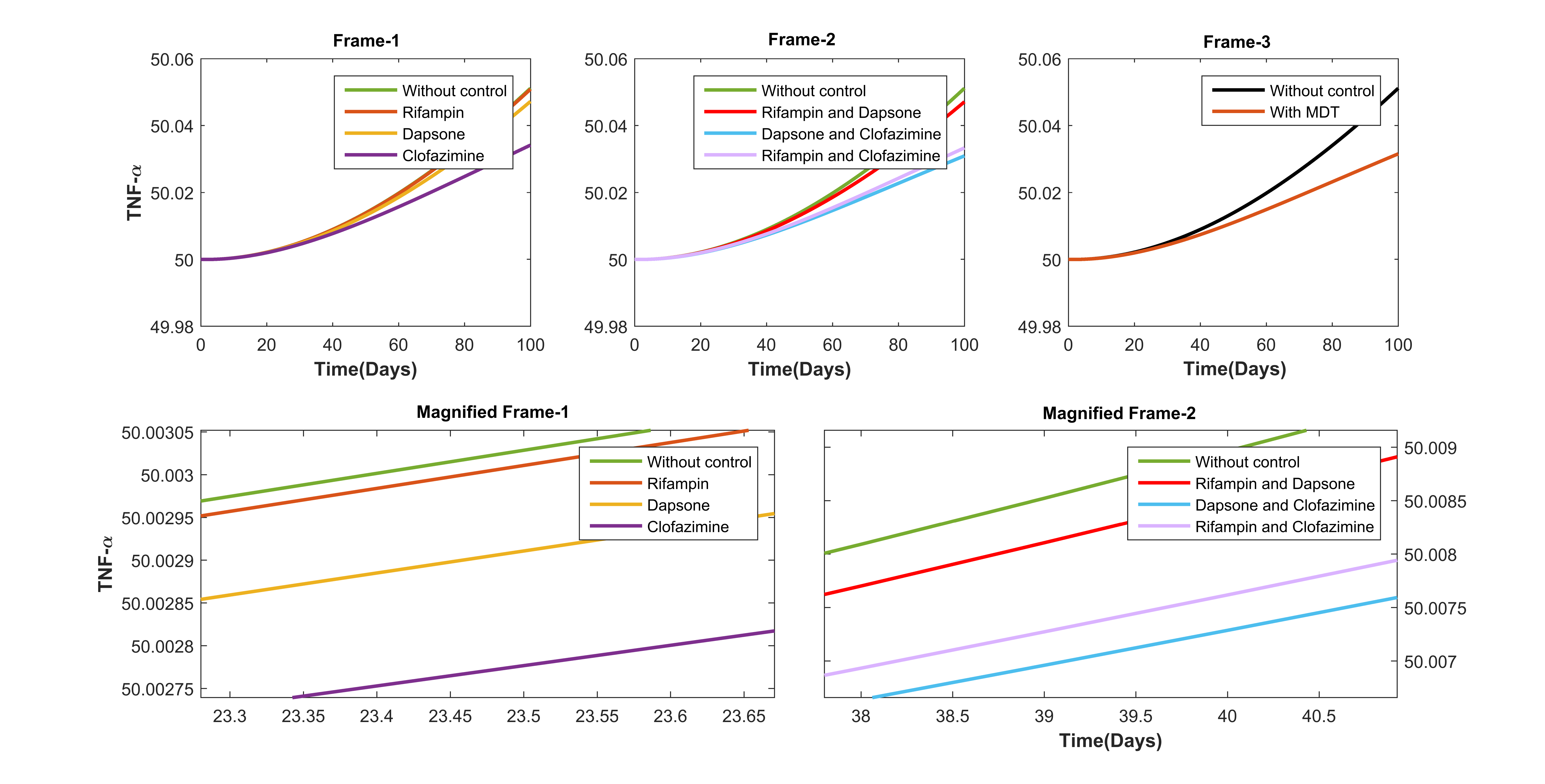}
    \caption{Plot depicting the dynamics of the  TNF-$\alpha$ on administration of MDT drugs individually (Frame-1), in combination of two (Frame-2) and all the three MDT drugs together (Frame-3). Magnified versions of Frame-1 and Frame-3 are also depicted for better clarity purpose to the reader.}
    \label{TA}
\end{figure} 
The figure \ref{TA} clearly depicts that the level of  TNF-$\alpha$  increases during \textit{Lepra reaction}. But the different combinations of drugs in MDT slow down the rate of increment of the level of TNF-$\alpha$. In case of individual administration of drugs, \textit{clofazimine} is the best for suppressing the increment of the levels of TNF-$\alpha$  followed by \textit{dapsone} and further followed by \textit{rifampin}. In case of combination of  two drugs  administration, \textit{dapsone} and   \textit{clofazimine} combination works  the best folowed by \textit{rifampin} and \textit{clofazimine} combination and further followed by \textit{rifampin} and \textit{dapsone} combination.  Similar behaviours can be observed for the cytokines $I L- 15, IL - 17 $ which are depicted in the figures \ref{IL15} and \ref{IL17} respectively. \\

\begin{figure}[ht!]
    \centering
    \includegraphics[height = 6cm, width =12cm]{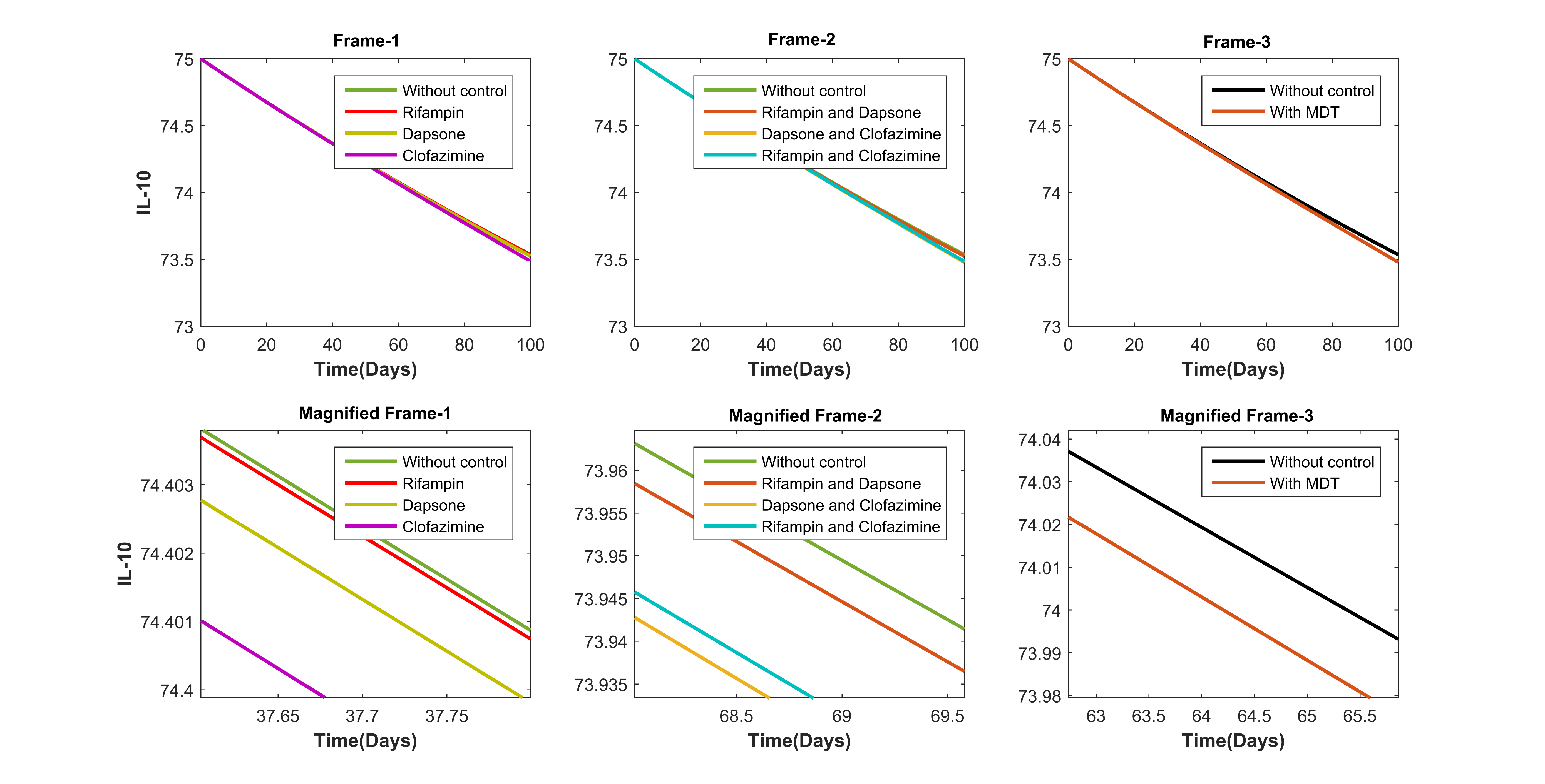}
    \caption{Plot depicting the dynamics of the $IL - 10$ on administration of MDT drugs individually (Frame-1), in combination of two (Frame-2) and all the three MDT drugs together (Frame-3). Magnified versions of Frame-1, Frame-2 and Frame-3 are also depicted for better clarity purpose to the reader.}
    \label{IL10}
\end{figure} 
\begin{figure}[ht!]
    \centering
    \includegraphics[height = 6cm, width =12cm]{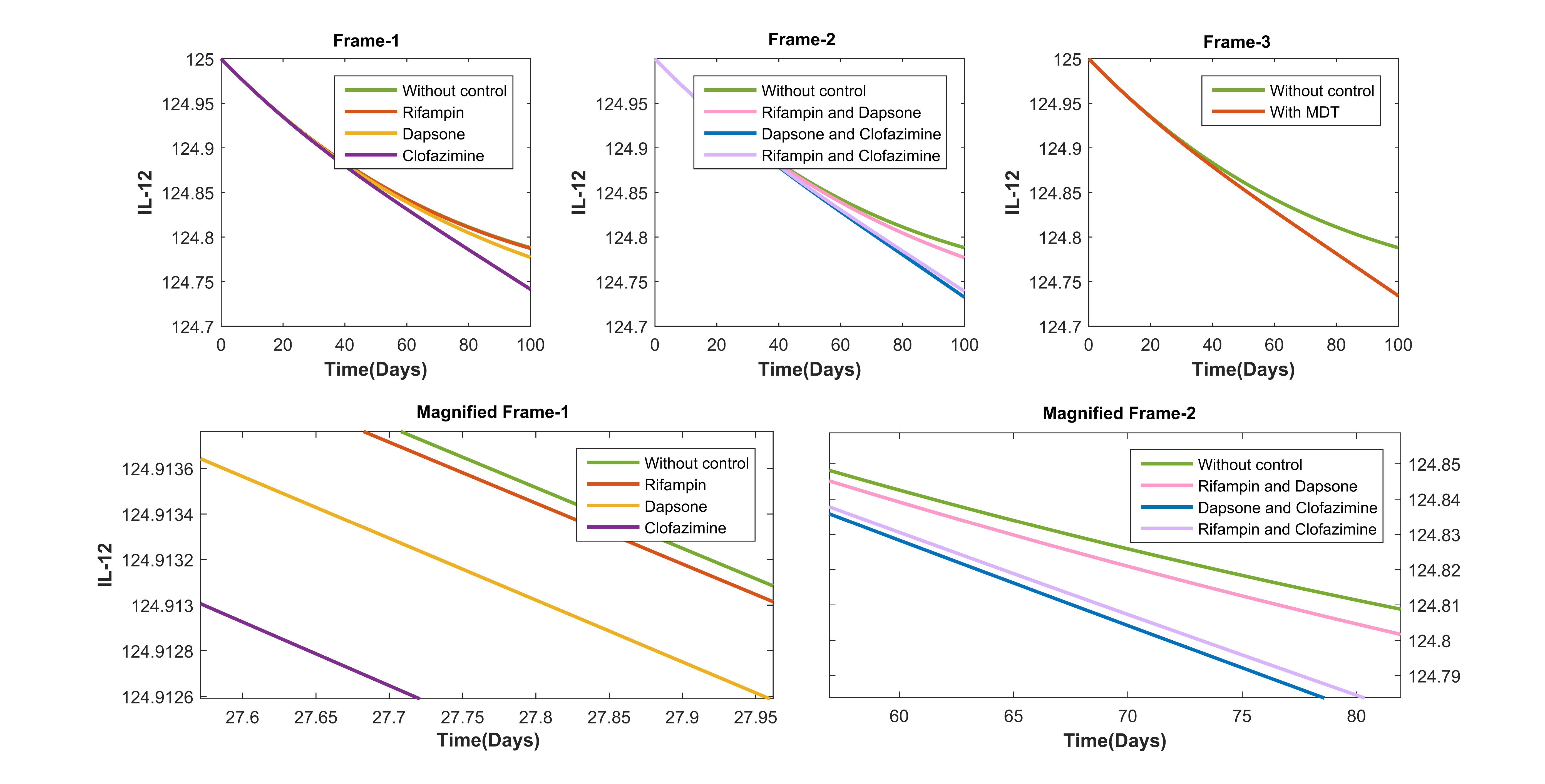}
    \caption{Plot depicting the dynamics of the $IL - 12$ on administration of MDT drugs individually (Frame-1), in combination of two (Frame-2) and all the three MDT drugs together (Frame-3). Magnified versions of Frame-1 and Frame-2 are also depicted for better clarity purpose to the reader.}
    \label{IL12}
\end{figure} 

From the   figures  \ref{IL10} and  \ref{IL12} it can be seen that during \textit{Lepra reaction} the levels  of both of the $IL-10$ and $IL-12$ cytokines decreases. But the different combinations of drug interventions of MDT can further enhance the rate of decrement. When drugs are applied individually, \textit{rifampin}  has the less impact in enhancing the decrement as compared to other drugs.  \textit{Clofazimine} has the most impact on enhancing the decrement. The degree of enhancement of the rate of decrement of the cytokines levels in case of two drugs combination follows the order, \textit{rifampin} and \textit{dapsone}  $<$ \textit{rifampin} and \textit{clofazimine} $<$ \textit{dapsone} and \textit{clofazimine}. Finally the the combination of three drugs does the same impact on cytokine level which can be see in the Frame-3 of the figures \ref{IL10} and \ref{IL12} respectively.

\begin{figure}[ht!]
    \centering
    \includegraphics[height = 6cm, width =12cm]{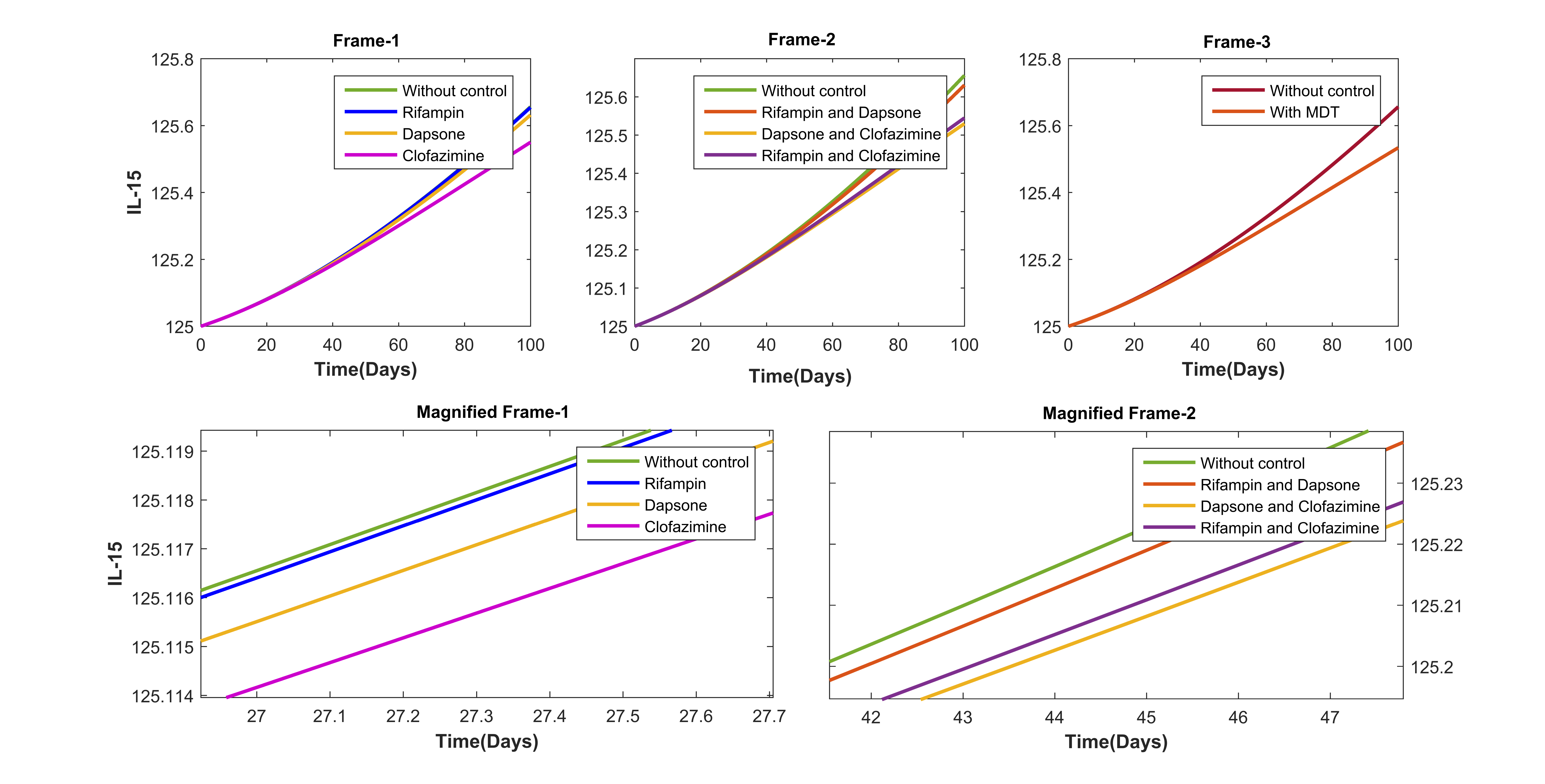}
    \caption{Plot depicting the dynamics of the $IL - 15$ on administration of MDT drugs individually (Frame-1), in combination of two (Frame-2) and all the three MDT drugs together (Frame-3). Magnified versions of Frame-1 and Frame-2 are also depicted for better clarity purpose to the reader.}
    \label{IL15}
\end{figure} 

\begin{figure}[ht!]
    \centering
    \includegraphics[height = 6cm, width =12cm]{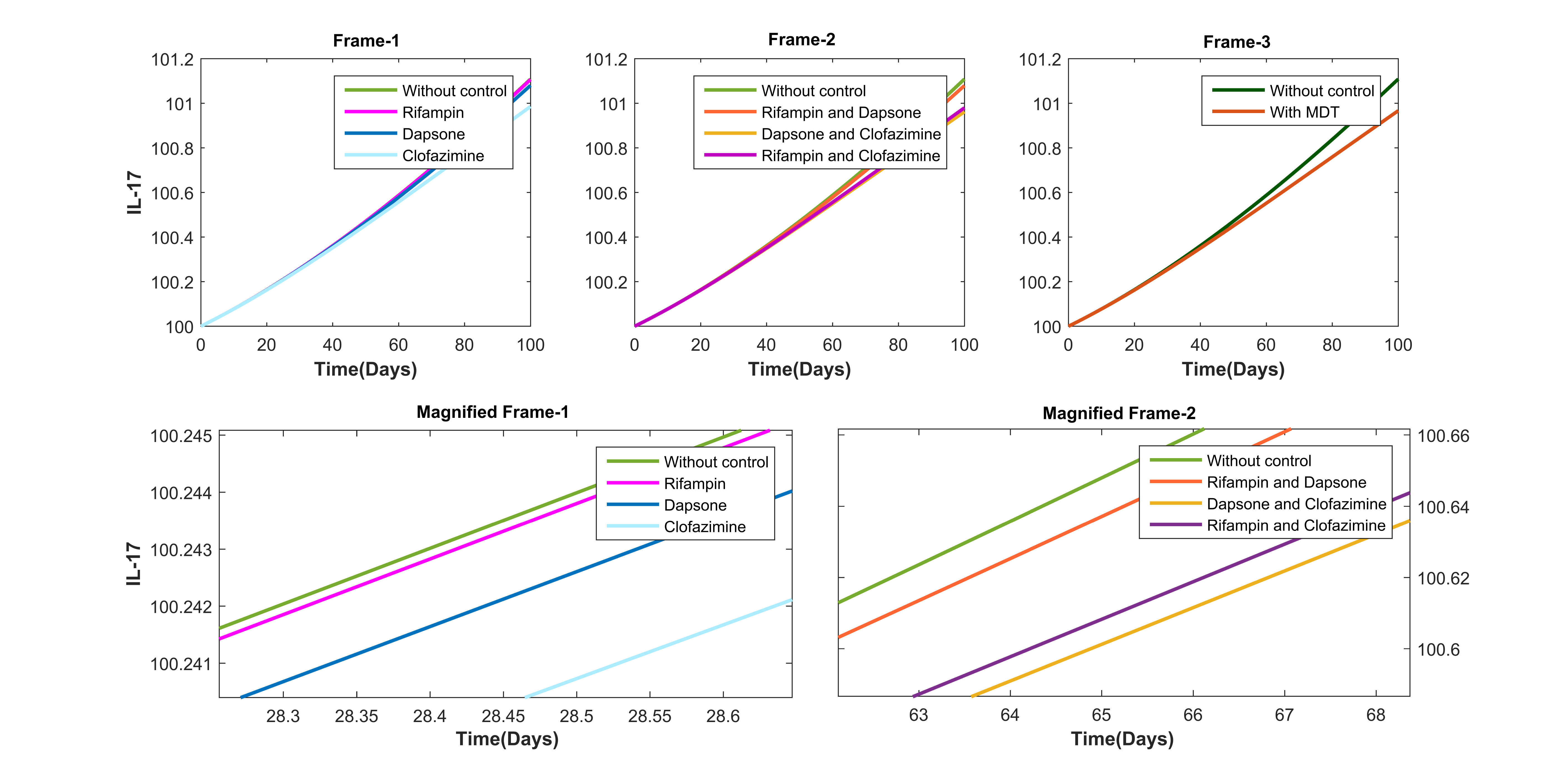}
    \caption{Plot depicting the dynamics of the $IL - 17$ on administration of MDT drugs individually (Frame-1), in combination of two (Frame-2) and all the three MDT drugs together (Frame-3). Magnified versions of Frame-1 and Frame-2 are also depicted for better clarity purpose to the reader.}
    \label{IL17}
\end{figure}

\section{Discussions and Conclusions}

The novel thing about  this paper is that it deals with a model that includes the dynamics of the levels of crucial cytokines that are involved in \textit{Lepra reaction} and also this work studies the impact of different drugs in MDT on the levels of these cytokines. The findings of the studies includes the following.

\begin{itemize}
    \item Among the drugs used in MDT for treating Leprosy, \textit{clofazimine} and \textit{dapsone} increas the susceptible cell count where as  \textit{rifampin} has an negative impact on it.
    
    \item The two drug combination of \textit{rifampin} and \textit{dapsone} has the negative impact on susceptible cells count.
    
    \item The MDT drug combinations decreases both the infected cell count as well as bacterial load.  \textit{Clofazimine} works the best in reduction when each of the drugs are administered individually and in combination of two drug administration, \textit{clofazimine} and \textit{dapsone} reduces the best.
    
    \item During the \textit{Lepra reaction} the levels of $IFN-\gamma$, $IL-10$ and $IL-12$ decreases whereas the levels  of $TNF-\alpha$, $IL-15$ and $IL-17$  increases.
    
    \item Each of the drugs used in MDT  enhances the $IFN-\gamma$ levels in host body. \textit{Clofazimine} enhances the best when each of the drugs are administered individually and in combination of two drug administration, \textit{clofazimine} and \textit{dapsone} enhances the best.
    
      \item The levels of both the cytokines  $IL-10$ and $IL-12$ decrease on administration  of the drugs in MDT. \textit{Rifampin} works the least in reduction when each of the drugs are administered individually and in combination of two drug administration, \textit{rifampin}- \textit{dapsone} combination impact the least.

    \item In the case of $TFN-\alpha$, $IL-15$ and $IL-17$ the drugs in MDT reduce the rate of increment of these bio markers.  \textit{Clofazimine}  is the best for suppressing the increment of these bio markers when each of the drugs are administered individually and in combination of two drug administration, \textit{clofazimine} and \textit{dapsone} turns out to be the best.

    \item In summary this is a novel and first of its kind work wherein we have discussed the natural history and dynamics of crucial bio markers in a \textit{Type-I Lepra reaction} and also studied in detail the influence of different combinations of drugs in MDT used for treating leprosy on the levels of these bio makers. This study can  be of important help to the clinician in early detection of the leprosy and avoid and control the disease from going to \textit{Lepra reactions} and help in averting major damages.
\end{itemize}

\printbibliography

\end{document}